\documentclass[twoside,a4paper]{amsart}
\PassOptionsToPackage{pdfauthor={Vladimir V. Kisil},%
    pdftitle={Continued Fractions, Moebius Transformations and Cycles},%
    pdfsubject={mathematics},%
    backref=page,%
  pdfkeywords={symmetries}
}{hyperref}
\usepackage{hyperref,graphicx}
\IfFileExists{eulervm.sty}{%
\usepackage{eulervm}%
}{}

\usepackage[initials,nobysame]{amsrefs}\relax

\BibSpecAlias{incollection}{inproceedings}
\BibSpec{article}{%
    +{}  {\PrintAuthors}                {author}
    +{,} { \emph}                            {title}
    +{.} { }                            {part}
    +{:} { }                            {subtitle}
    +{.} { \PrintContributions}         {contribution}
    +{.} { \PrintPartials}              {partial}
    +{.} { }                       {journal}
    +{,} { }                            {volume}
    +{} { \parenthesize}                 {date}
    +{, No.}  { }               {number}
    +{, } {}                             {pages}
    +{,} { }                            {status}
    +{.} { \PrintTranslation}           {translation}
    +{.} { Reprinted in \PrintReprint}  {reprint}
    +{.} { }                            {note}
    +{.} {}                             {transition}
}

\BibSpec{partial}{%
    +{}  {}                             {part}
    +{:} { }                            {subtitle}
    +{.} { \PrintContributions}         {contribution}
    +{.} { \emph}                       {journal}
    +{,} { \textbf}                            {volume}
    +{}  { \parenthesize}               {number}
    +{:} {}                             {pages}
    +{,} { \PrintDateB}                 {date}
}

\BibSpec{book}{%
    +{}  {\PrintPrimary}                {transition}
    +{.} { \emph}                       {title}
    +{.} { }                            {part}
    +{:} { \emph}                       {subtitle}
    +{.} { }                            {series}
    +{,} { \voltext}                    {volume}
    +{.} { Edited by \PrintNameList}    {editor}
    +{.} { Translated by \PrintNameList}{translator}
    +{.} { \PrintContributions}         {contribution}
    +{.} { }                            {publisher}
    +{.} { }                            {organization}
    +{,} { }                            {address}
    +{,} { \PrintEdition}               {edition}
    +{,} { \PrintDateB}                 {date}
    +{.} { }                            {note}
    +{.} {}                             {transition}
    +{.} { \PrintTranslation}           {translation}
    +{.} { Reprinted in \PrintReprint}  {reprint}
    +{.} {}                             {transition}
}

\BibSpec{collection.article}{%
    +{}  {\PrintAuthors}                {author}
    +{.} { \emph}                            {title}
    +{.} { }                            {part}
    +{:} { }                            {subtitle}
    +{.} { \PrintContributions}         {contribution}
    +{.} { \PrintConference}            {conference}
    +{.} { \PrintBook}                  {book}
    +{.} { In \PrintEditorsA}              {editor}
    +{.} { }                         {booktitle}
    +{,} { pages~}                      {pages}
    +{,} { }                            {publisher}
    +{,} { }                            {organization}
    +{,} { }                            {address}
    +{,} { \PrintDateB}                 {date}
    +{.} { \PrintTranslation}           {translation}
    +{.} { Reprinted in \PrintReprint}  {reprint}
    +{.} { }                            {note}
    +{.} {}                             {transition}
}

 \BibSpec{inproceedings}{%
     +{}  {\PrintAuthors}                {author}
     +{.} { \emph}                            {title}
     +{.} { }                            {part}
     +{:} { }                            {subtitle}
     +{.} { \PrintContributions}         {contribution}
     +{.} { \PrintConference}            {conference}
     +{.} { \PrintBook}                  {book}
     +{.} { In \PrintEditorsA}              {editor}
     +{.} {  }                         {booktitle}
     +{,} { pages~}                      {pages}
     +{,} { }                            {publisher}
     +{,} { }                            {organization}
     +{,} { }                            {address}
     +{,} { \PrintDateB}                 {date}
     +{.} { \PrintTranslation}           {translation}
     +{.} { Reprinted in \PrintReprint}  {reprint}
     +{.} { }                            {note}
     +{.} {}                             {transition}
 }

\BibSpec{conference}{%
    +{}  {}                        {title}
    +{}  {\PrintConferenceDetails} {transition}
}

\BibSpec{innerbook}{%
    +{.} { \emph}                       {title}
    +{.} { }                            {part}
    +{:} { \emph}                       {subtitle}
    +{.} { }                            {series}
    +{,} { \voltext}                    {volume}
    +{.} { Edited by \PrintNameList}    {editor}
    +{.} { Translated by \PrintNameList}{translator}
    +{.} { \PrintContributions}         {contribution}
    +{.} { }                            {publisher}
    +{.} { }                            {organization}
    +{,} { }                            {address}
    +{,} { \PrintEdition}               {edition}
    +{,} { \PrintDateB}                 {date}
    +{.} { }                            {note}
    +{.} {}                             {transition}
}

\BibSpec{report}{%
    +{}  {\PrintPrimary}                {transition}
    +{.} { \emph}                       {title}
    +{.} { }                            {part}
    +{:} { \emph}                       {subtitle}
    +{.} { \PrintContributions}         {contribution}
    +{.} { Technical Report }           {number}
    +{,} { }                            {series}
    +{.} { }                            {organization}
    +{,} { }                            {address}
    +{,} { \PrintDateB}                 {date}
    +{.} { \PrintTranslation}           {translation}
    +{.} { Reprinted in \PrintReprint}  {reprint}
    +{.} { }                            {note}
    +{.} {}                             {transition}
}

\BibSpec{thesis}{%
    +{}  {\PrintAuthors}                {author}
    +{,} { \emph}                       {title}
    +{:} { \emph}                       {subtitle}
    +{.} { \PrintThesisType}            {type}
    +{.} { }                            {organization}
    +{,} { }                            {address}
    +{,} { \PrintDateB}                 {date}
    +{.} { \PrintTranslation}           {translation}
    +{.} { Reprinted in \PrintReprint}  {reprint}
    +{.} { }                            {note}
    +{.} {}                             {transition}
}

\newtheorem{thm}{Theorem}[section]
\newtheorem{claim}[thm]{Claim}
\newtheorem{lem}[thm]{Lemma}

\providecommand{\Cliff}[2][\comment]{{\ensuremath{%
\mathcal{C}\kern-0.13em\ell(#1,#2)}}}%

\providecommand{\comment}[1]{}
\providecommand{\norm}[2][\relax]{\left\|#2\right\|\ifx#1\relax\else_{#1}\fi}
\providecommand{\modulus}[2][\relax]{\left| #2 \right|\ifx#1\relax\else_{#1}\fi}
\providecommand{\cycle}[3][]{{#1 C^{#2}_{#3}}}
\providecommand{\Space}[3][]{\ensuremath{\mathbb{#2}^{#3}_{#1}{}}}
  \providecommand{\FSpace}[3][]{\ensuremath{\ifx#2l \ell_{#3}^{#1}{}\else
  #2_{#3}^{#1}{}\fi}} 
\providecommand{\scalar}[3][\relax]{\left\langle #2,#3 
        \right\rangle\ifx#1\relax\else_{#1}\fi}
  \providecommand{\Zbl}[1]{Zbl\href{http://www.emis.de:80/cgi-bin/zmen/ZMATH/en/zmathf.html?first=1&maxdocs=3&type=html&an=#1&format=complete}{#1}}
\providecommand{\myeprint}[2]{E-print: \href{#1}{\texttt{#2}}}

\providecommand{\rmi}{\mathrm{i}}
\providecommand{\tr}{\mathop{tr}}
  \PII{\relax}
  \copyrightinfo{}{\relax}

\begin{document}
\title[Continued Fractions, M\"obius Transformations and Cycles]
{Remark on Continued Fractions,\\
  M\"obius Transformations and Cycles}

\author[Vladimir V. Kisil]%
{\href{http://www.maths.leeds.ac.uk/~kisilv/}{Vladimir V. Kisil}}
\thanks{On  leave from Odessa University.}

\address{%
School of Mathematics\\
University of Leeds\\
Leeds LS2\,9JT\\
UK
}

\email{\href{mailto:kisilv@maths.leeds.ac.uk}{kisilv@maths.leeds.ac.uk}}

\urladdr{\url{http://www.maths.leeds.ac.uk/~kisilv/}}

\date{\today}

\begin{abstract}
  We review interrelations between continued fractions, M\"obius
  transformations and representations of cycles by \(2\times 2\)
  matrices. This leads us to several descriptions of continued
  fractions through chains of orthogonal or touching horocycles. One
  of these descriptions was proposed in a recent paper by A.~Beardon
  and I.~Short. The approach is extended to several dimensions in a
  way which is compatible to the early propositions of A.~Beardon based
  on Clifford algebras.
\end{abstract}
\maketitle
\tableofcontents


\section{Introduction}
\label{sec:introduction}

Continued fractions remain an important and attractive topic of
current research \citelist{\cite{Khrushchev08a}
  \cite{BorweinPoortenShallitZudilin14a} \cite{Karpenkov2013a}
  \cite{Kirillov06}*{\S~E.3}}. A fruitful and geometrically appealing
method considers a continued fraction as an (infinite) product of
linear-fractional transformations from the M\"obius group.  see
Sec.~\ref{sec:continued-fractions} of this paper for an overview,
papers~\citelist{\cite{PaydonWall42a} \cite{Schwerdtfeger45a}
  \cite{PiranianThron57a} \cite{Beardon04b}
  \cite{Schwerdtfeger79a}*{Ex.~10.2}} and in particular
\cite{Beardon01a} contain further references and historical
notes. Partial products of linear-fractional maps form a sequence in
the Moebius group, the corresponding sequence of transformations can
be viewed as a discrete dynamical
system~\cite{Beardon01a,MageeOhWinter14a}. Many important questions on
continued fractions, e.g. their convergence, can be stated in terms of
asymptotic behaviour of the associated dynamical system. Geometrical
generalisations of continued fractions to many dimensions were
introduced recently as well~\citelist{\cite{Beardon03a}
  \cite{Karpenkov2013a}}.

Any consideration of the M\"obius group introduces cycles---the
M\"obius invariant set of all circles and straight lines. Furthermore, 
an efficient treatment cycles and M\"obius
transformations is realised through certain  \(2\times 2\) matrices,
which we will review in Sec.~\ref{sec:mobi-transf-cycl}, see
also~\citelist{\cite{Schwerdtfeger79a} \cite{Cnops02a}*{\S~4.1}
  \cite{FillmoreSpringer90a} \cite{Kirillov06}*{\S~4.2}
  \cite{Kisil06a} \cite{Kisil12a}}. Linking the above relations we may
propose the main thesis of the present note:
\begin{claim}[Continued fractions and cycles]
  Properties of continued fractions may be illustrated and
  demonstrated using related cycles, in particular, in the form of
  respective \(2\times 2\) matrices. 
\end{claim}
One may expect that such an observation has been made a while ago,
e.g. in the book~\cite{Schwerdtfeger79a}, where both topics were
studied. However, this seems did not happened for some reasons. It is
only the recent paper~\cite{BeardonShort14a}, which pioneered a connection
between continued fractions and cycles. We hope that the explicit
statement of the claim will stimulate its further fruitful
implementations. In particular, Sec.~\ref{sec:cont-fract-cycl} reveals
all three possible cycle arrangements similar to one used
in~\cite{BeardonShort14a}. Secs.~\ref{sec:multi-dimens-vari}--\ref{sec:mult-cont-fract}
shows that relations between continued fractions and cycles can be
used in the multidimensional case as well.

As an illustration, we draw on Fig.~\ref{fig:cont-fract-e} chains of
tangent horocycles (circles tangent to the real line,
see~\cite{BeardonShort14a} and Sec.~\ref{sec:cont-fract-cycl}) for two
classical simple continued fractions:
\begin{displaymath}
  e=2+
  \cfrac{1}{1+\cfrac{1}{2+\cfrac{1}{1+\cfrac{1}{1+\ldots}}}}\,
  ,\qquad
  \pi=3+
  \cfrac{1}{7+\cfrac{1}{15+\cfrac{1}{1+\cfrac{1}{292+\ldots}}}}\, .
\end{displaymath}
\begin{figure}[htbp]
  \centering
  \includegraphics[scale=.6]{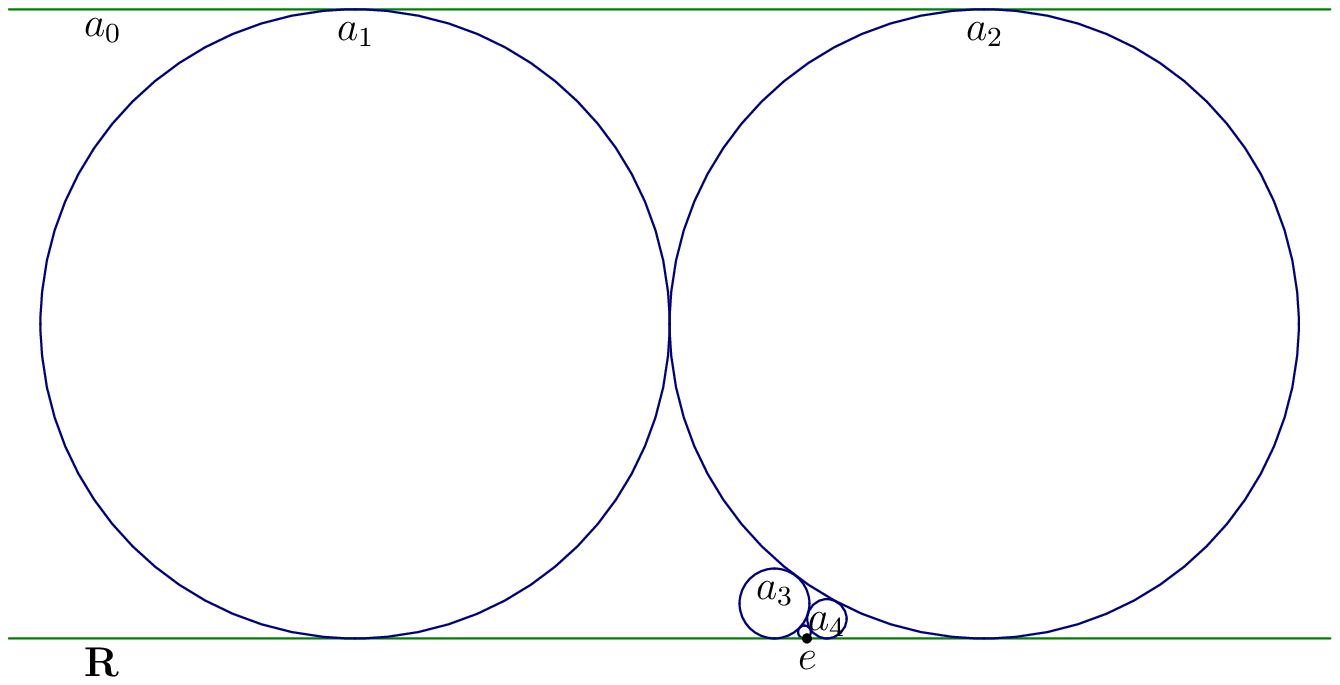}\hfill
  \includegraphics[scale=.6]{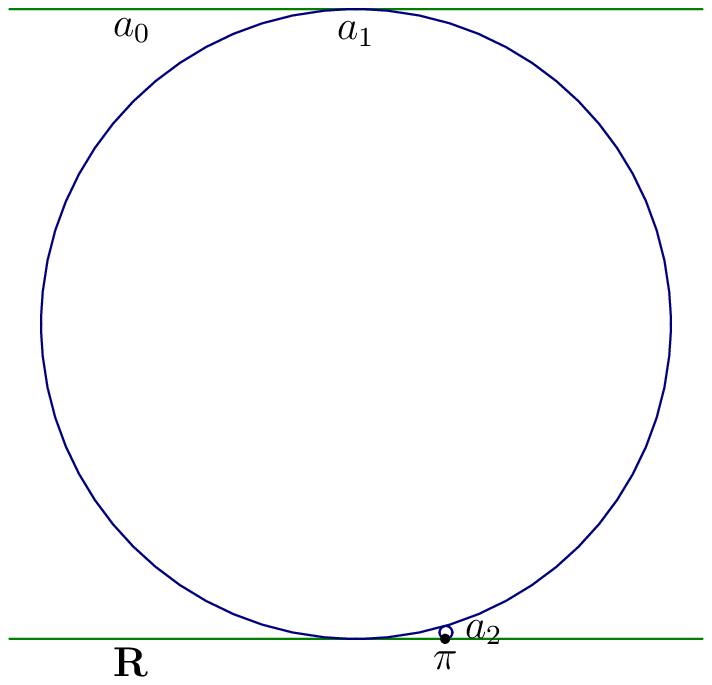}
  \caption{Continued fractions for $e$ and $\pi$ visualised. The
    convergence rate for $\pi$ is pictorially faster.}
  \label{fig:cont-fract-e}
\end{figure}
One can immediately see, that the convergence for \(\pi\) is much
faster: already the third horocycle is too small to be visible even if
it is drawn. This is related to larger coefficients in the continued
fraction for \(\pi\). 

Paper~\cite{BeardonShort14a} also presents examples of proofs based on
chains of horocycles.  This intriguing construction was introduced
in~\cite{BeardonShort14a} \emph{ad hoc}. Guided by the above claim we
reveal sources of this and other similar arrangements of horocycles.
Also, we can produce multi-dimensional variant of the framework.

\section{Continued Fractions}
\label{sec:continued-fractions}

We use the following compact notation for a continued fraction:
\begin{equation}
\label{eq:cont-frac}
K(a_n|b_n)=
\cfrac{a_1}{b_1+\cfrac{a_2}{b_2+\cfrac{a_3}{b_3+\ldots}}}=
\frac{a_1}{b_1}\genfrac{}{}{0pt}{}{}{+}\frac{a_2}{b_2}\genfrac{}{}{0pt}{}{}{+}
\frac{a_3}{b_3}\genfrac{}{}{0pt}{}{}{+\ldots}.
\end{equation}
Without loss of generality we can assume \(a_j\neq 0\) for all
\(j\). 
The important particular case of \emph{simple} continued fractions,
with \(a_n=1\) for all \(n\), is denoted by \(K(b_n)=K(1|b_n)\). Any
continued fraction can be transformed to an equivalent simple one.

It is easy to see, that continued fractions are
related to the following linear-fractional
(M\"obius) transformation, cf.~\cites{PaydonWall42a,Schwerdtfeger45a,PiranianThron57a,Beardon04b}:
\begin{equation}
  \label{eq:cf-moebius-maps}
  S_n=s_1\circ s_2\circ \ldots\circ s_n, \qquad \text{where} \quad
  s_{j}(z)= \frac{a_j}{b_j+z}.
\end{equation}
These M\"obius transformation are considered as bijective maps of the
Riemann sphere \(\dot{\Space{C}{}}=\Space{C}{}\cup\{\infty\}\) onto itself.
If we associate  the matrix \(
\begin{pmatrix}
  a&b\\ c&d 
\end{pmatrix}\) to a liner-fractional transformation \(z\mapsto
\frac{a z+b}{c z+d}\), then the composition of two
such transformations corresponds to multiplication of the
respective matrices. Thus, relation~\eqref{eq:cf-moebius-maps}
has the matrix form:
\begin{equation}
  \label{eq:cf-part-frac-matrix}
  \begin{pmatrix}
    P_{n-1}&P_n\\
    Q_{n-1}&Q_n
  \end{pmatrix}
  =
  \begin{pmatrix}
    0&a_1\\
    1&b_1
  \end{pmatrix}
  \begin{pmatrix}
    0&a_2\\
    1&b_2
  \end{pmatrix}\cdots
  \begin{pmatrix}
    0&a_n\\
    1&b_n
  \end{pmatrix} .
\end{equation}
The last identity can be fold into the recursive formula: 
\begin{equation}
  \label{eq:cf-recurr-matrix}
  \begin{pmatrix}
    P_{n-1}&P_{n}\\    
    Q_{n-1}&Q_{n}
  \end{pmatrix}
  =
  \begin{pmatrix}
    P_{n-2}&P_{n-1}\\    
    Q_{n-2}&Q_{n-1}
  \end{pmatrix}
  \begin{pmatrix}
    0&a_n\\
    1&b_n
  \end{pmatrix}.
\end{equation}
This is equivalent to the main recurrence relation:
\begin{equation}
  \label{eq:cf-recurr}
  \begin{array}{c}
    P_n=b_nP_{n-1}+a_nP_{n-2}\\
    Q_n=b_nQ_{n-1}+a_nQ_{n-2}
  \end{array}, 
\quad n=1,2,3,\ldots,
\quad\text{with }
\begin{array}{cc}
{P_1}={a_1},&  {P_{0}}=0,\\
{Q_1}={b_1}, & Q_{0}=1.
\end{array}
\end{equation}

The meaning of entries \(P_n\) and \(Q_n\) from the
matrix~\eqref{eq:cf-part-frac-matrix} is revealed as follows. M\"obius
transformation~\eqref{eq:cf-moebius-maps}--\eqref{eq:cf-part-frac-matrix}
maps \(0\) and \(\infty\) to
\begin{equation}
  \label{eq:part-frac-moebius-map}
  \frac{P_n}{Q_n}=S_n(0), \qquad     \frac{P_{n-1}}{Q_{n-1}}=S_n(\infty).
\end{equation}
It is easy to see that \(S_n(0)\) is the \emph{partial quotient} of~\eqref{eq:cont-frac}:
\begin{equation}
  \label{eq:partial-quotient}
  \frac{P_n}{Q_n}= 
  \frac{a_1}{b_1}\genfrac{}{}{0pt}{}{}{+}\frac{a_2}{b_2}\genfrac{}{}{0pt}{}{}{+\ldots+}
\frac{a_n}{b_n}\,.
\end{equation}
Properties of the sequence of partial quotients
\(\left\{\frac{P_n}{Q_n}\right\}\) in terms of sequences \(\{a_n\}\)
and \(\{b_n\}\) are the core of the continued fraction
theory. Equation~\eqref{eq:part-frac-moebius-map} links partial
quotients with the M\"obius map~\eqref{eq:cf-moebius-maps}.  Circles
form an invariant family under M\"obius transformations, thus their
appearance for continued fractions is natural. Surprisingly, this
happened only recently in~\cite{BeardonShort14a}.

\section{M\"obius Transformations and Cycles}
\label{sec:mobi-transf-cycl}

If \(M=\begin{pmatrix}
  a&b\\c&d
\end{pmatrix}\) is a matrix with real entries then for the purpose of
the associated M\"obius transformations \(M: z\mapsto
\frac{az+b}{cz+d}\) we may assume that \(\det M=\pm 1\). The
collection of all such matrices form a group.  M\"obius maps commute
with the complex conjugation \(z\mapsto \bar{z}\). If \(\det M>0\)
then both the upper and the lower half-planes are preserved; if \(\det
M <0\) then the two half-planes are swapped. Thus, we can treat \(M\)
as the map of equivalence classes \(z\sim\bar{z}\), which are labelled
by respective points of the closed upper half-plane. Under this
identification we consider any map produced by \(M\) with \(\det M
=\pm 1\) as the map of the closed upper-half plane to itself.

The characteristic property of M\"obius maps is that circles and lines
are transformed to circles and lines. We use the word \emph{cycles}
for elements of this M\"obius-invariant
family~\cites{Yaglom79,Kisil12a,Kisil06a}. We abbreviate a cycle given
by the equation
\begin{equation}
  \label{eq:cycle-defn}
  k(u^2+v^2)-2lv-2nu+m=0
\end{equation}
to the point \((k,l,n,m)\) of the three dimensional projective space
\(P\Space{R}{3}\). The equivalence relation \(z\sim \bar{z}\) is
lifted  to the equivalence relation 
\begin{equation}
  \label{eq:cycle-equiv}
  (k,l,n,m)\sim(k,l,-n,m)
\end{equation}
in the space of cycles, which again is compatible with the M\"obius
transformations acting on cycles.

The most efficient connection between cycles and M\"obius
transformations is realised through the construction, which may be
traced back to~\cite{Schwerdtfeger79a} and was subsequently
rediscovered by various authors~\citelist{\cite{Cnops02a}*{\S~4.1}
  \cite{FillmoreSpringer90a} \cite{Kirillov06}*{\S~4.2}}, see
also~\cites{Kisil06a,Kisil12a}. The construction associates a cycle
\((k,l,n,m)\) with the \(2\times 2\) matrix \(\cycle{}{}=\begin{pmatrix}
  l+\rmi  n&-m\\
  k&-l+\rmi n
\end{pmatrix}\), see discussion in~\cite{Kisil12a}*{\S~4.4} for a
justification. This identification is M\"obius covariant: the M\"obius
transformation defined by \(M=\begin{pmatrix} a&b\\c&d
\end{pmatrix}\) maps a cycle with matrix \(\cycle{}{}\) to the cycle with
matrix \(M\cycle{}{}M^{-1}\). Therefore, any M\"obius-invariant relation between cycles
can be expressed in terms of corresponding matrices. The central role
is played by the M\"obius-invariant inner pro\-duct~\cite{Kisil12a}*{\S~5.3}: 
\begin{equation}
  \label{eq:cycle-iner-pr}
  \scalar{\cycle{}{}}{\cycle[      \tilde]{}{}}=\Re \tr(\cycle{}{}\overline{\cycle[\tilde]{}{}}),
\end{equation}
which is a cousin of the product used in GNS construction of
\(C^*\)-algebras. Notably, the relation:
\begin{equation}
  \label{eq:first-ortho-gen}
  \scalar{\cycle{}{}}{\cycle[\tilde]{}{}}=0\quad \text{ or } \quad  k\tilde{m}+m\tilde{k}-2n\tilde{n}-2l\tilde{l}=0,
\end{equation}
describes two cycles \(\cycle{}{}=(k,l,m,n)\) and
\(\cycle[\tilde]{}{}=(\tilde{k},\tilde{l},\tilde{m},\tilde{n})\)
orthogonal in Euclidean geometry.  Also, the inner
product~\eqref{eq:cycle-iner-pr} expresses the Descartes--Kirillov
condition~\citelist{\cite{Kirillov06}*{Lem.~4.1(c)}
  \cite{Kisil12a}*{Ex.~5.26}} of \(\cycle{}{}\) and
\(\cycle[\tilde]{}{}\) to be externally tangent:
\begin{equation}
  \label{eq:tangent-cycles}
  \scalar{\cycle{}{}+\cycle[\tilde]{}{}}{\cycle{}{}+\cycle[\tilde]{}{}} =0 
  \quad\text{or}\quad
  (l+\tilde{l})^2+(n+\tilde{n})^2 -(m+\tilde{m})(k+\tilde{k})=0,
\end{equation}
where the representing vectors \(\cycle{}{}=(k,l,n,m)\) and
\(\cycle[\tilde]{}{}=(\tilde{k},\tilde{l},\tilde{m},\tilde{n})\) from
\(P\Space{R}{3}\) need to be normalised by the conditions
\(\scalar{\cycle{}{}}{\cycle{}{}}=1\) and
\(\scalar{\cycle[\tilde]{}{}}{\cycle[\tilde]{}{}}=1\).

\section{Continued Fractions and Cycles}
\label{sec:cont-fract-cycl}

Let \(M=\begin{pmatrix}
  a&b\\c&d
\end{pmatrix}\) be a matrix with real entries and the determinant
\(\det M\) equal to \(\pm 1\), we denote this by \(\delta=\det M\). As
mentioned in the previous section, to calculate the image of a cycle
\(\cycle{}{}\) under M\"obius transformations \(M\) we can use matrix
similarity \(M\cycle{}{}M^{-1}\). If \(M=
\begin{pmatrix}
  P_{n-1}&P_n\\
  Q_{n-1}&Q_n 
\end{pmatrix}\) is the
matrix~\eqref{eq:cf-part-frac-matrix} associated to a continued
fraction and we are interested in the partial fractions
\(\frac{P_n}{Q_n}\), it is natural to ask: 
\begin{quote}
  \emph{Which cycles \(\cycle{}{}\) have transformations
    \(M\cycle{}{}M^{-1}\) depending on the first (or on the second)
    columns of \(M\) only?}
\end{quote}

It is a straightforward
calculation with matrices\footnote{This calculation can be done with
  the help of the tailored Computer Algebra System (CAS) as described
  in~\citelist{\cite{Kisil12a}*{App.~B}\cite{Kisil05b}}.} to check the
following statements: 

\begin{lem}
  \label{le:first-col}
  The cycles \((0,0,1,m)\) (the horizontal lines \(v=m\)) are the only
  cycles, such that their
  images under the M\"obius transformation \(\begin{pmatrix}
    a&b\\c&d
  \end{pmatrix}\) are independent from the column \(
  \begin{pmatrix}
    b\\d
  \end{pmatrix}\). The image 
  associated to the column \(\begin{pmatrix} a\\c
  \end{pmatrix}\) is the horocycle \((c^2m,acm,\delta,a^2m)\), which
  touches the real line at \(\frac{a}{c}\) and has the radius
  \(\frac{1 }{mc^2}\).
\end{lem}
In particular, for the matrix~\eqref{eq:cf-recurr-matrix} the
horocycle is touching the real line at the point
\(\frac{P_{n-1}}{Q_{n-1}}=S_n(\infty)\)~\eqref{eq:part-frac-moebius-map}. 
\begin{lem}
  \label{le:second-col}
  The cycles \((k,0,1,0)\) (with the equation \(k(u^2+v^2)-2v=0\)) are
  the only cycles, such that their images under the M\"obius
  transformation \(\begin{pmatrix} a&b\\c&d
  \end{pmatrix}\) are independent from the column \(
  \begin{pmatrix}
    a\\c
  \end{pmatrix}\). The image 
  associated to the column  \(\begin{pmatrix} b\\d
  \end{pmatrix}\) is the horocycle
  \((d^2k,bdk,\delta,b^2k)\), which touches the real line at
  \(\frac{b}{d}\) and has the radius \(\frac{1 }{kd^2}\).
\end{lem}
 In particular, for the matrix~\eqref{eq:cf-recurr-matrix} the horocycle is
touching the real line at the point
\(\frac{P_n}{Q_n}=S_n(0)\)~\eqref{eq:part-frac-moebius-map}. 
In view of these partial quotients the following cycles joining them
are of interest.
\begin{lem}
  \label{le:blue}
  A cycle \((0,1,n,0)\) (any non-horizontal line passing
  \(0\)) is transformed by
  \eqref{eq:cf-moebius-maps}--\eqref{eq:cf-part-frac-matrix} to the
  cycle \((2Q_nQ_{n-1}, P_nQ_{n-1}+Q_nP_{n-1},\delta n,2P_nP_{n-1})\),
  which passes points \(\frac{P_n}{Q_n}=S_n(0)\) and
  \(\frac{P_{n-1}}{Q_{n-1}}=S(\infty)\) on the real line.
\end{lem}
The above three families contain cycles with specific relations to
partial quotients through M\"obius transformations. There is one
degree of freedom in each family: \(m\), \(k\) and \(n\),
respectively. We can use the parameters to create an ensemble of three
cycles (one from each family) with some M\"obius-invariant
interconnections. Three most natural arrangements are illustrated by
Fig.~\ref{fig:vari-arrang-three}.  The first row presents the initial
selection of cycles, the second row---their images after a M\"obius
transformation (colours are preserved). The arrangements are as
follows:
\begin{figure}[htbp]
  \centering
  \includegraphics{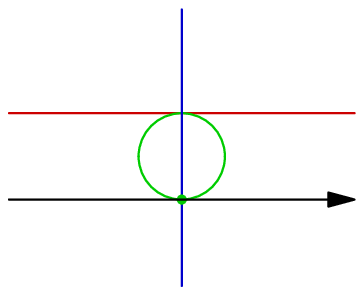}\hfil
  \includegraphics{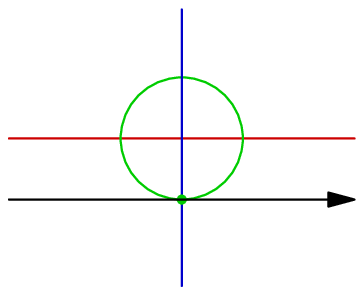}\hfil
  \includegraphics{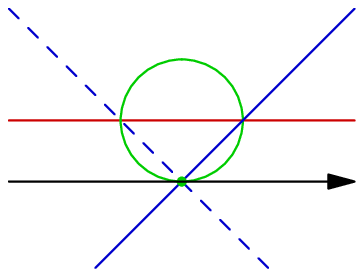}
\\
  \includegraphics{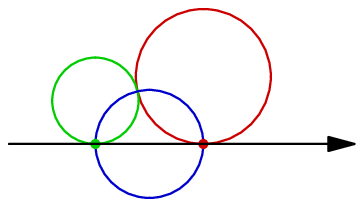}\hfil
  \includegraphics{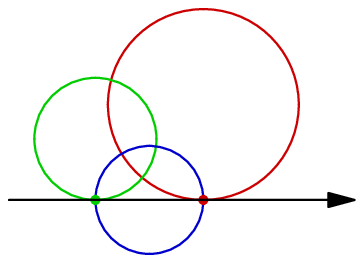}\hfil
  \includegraphics{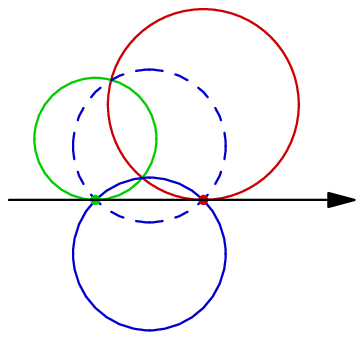}
  \caption[Various arrangements for three cycles]{Various arrangements
    for three cycles. The first row shows the initial position, the
    second row---after a M\"obius
    transformation (colours are preserved). \\
    The left column shows the arrangement used in the
    paper~\cite{BeardonShort14a}: two horocycles touching, the connecting
    cycle is passing their common point
    and is orthogonal to the real line.\\
    The central column presents two orthogonal horocycles and the
    connecting
    cycle orthogonal to them. \\
    The horocycles in the right column are again orthogonal but the
    connecting cycle passes one of their intersection points and makes
    \(45^\circ\) with the real axis.}
  \label{fig:vari-arrang-three}
\end{figure}

\begin{enumerate}
\item The left column shows the arrangement used in the
  paper~\cite{BeardonShort14a}: two horocycles are tangent, the third
  cycle, which we call \emph{connecting}, passes three points of
  pair-wise contact between horocycles and the real line. The
  connecting cycle is also orthogonal to horocycles and the real
  line. The arrangement corresponds to the following values \(m=2\),
  \(k=2\), \(n=0\). These parameters are uniquely defined by the above
  tangent and orthogonality conditions together with the requirement
  that the horocycles' radii agreeably depend from the consecutive
  partial quotients' denominators: \(\frac{1}{2Q_{n-1}^2}\) and
  \(\frac{1}{2Q_{n}^2}\) respectively. This follows from the explicit
  formulae of image cycles calculated in Lemmas~\ref{le:first-col}
  and~\ref{le:second-col}.

\item The central column of Fig.~\ref{fig:vari-arrang-three} presents
  two orthogonal horocycles and the connecting cycle orthogonal to
  them. The initial cycles have parameters \(m=\sqrt{2}\),
  \(k=\sqrt{2}\), \(n=0\). Again, these values follow from the
  geometric conditions and the alike dependence of radii from
  the partial quotients' denominators: \(\frac{\sqrt{2}}{2Q_{n-1}^2}\) and
  \(\frac{\sqrt{2}}{2Q_{n}^2}\).

\item Finally, the right column have the same two orthogonal
  horocycles, but the connecting cycle passes one of two horocycles'
  intersection points. Its mirror reflection in the real
  axis satisfying~\eqref{eq:cycle-equiv} (drawn in the dashed style) passes
  the second intersection point.  This corresponds to the values
  \(m=\sqrt{2}\), \(k=\sqrt{2}\), \(n=\pm 1\). The connecting cycle makes
  the angle \(45^\circ\) at the points of intersection with the real
  axis. It also has the radius
  \(\frac{\sqrt{2}}{2}\modulus{\frac{P_n}{Q_n}-\frac{P_{n-1}}{Q_{n-1}}}
  =\frac{\sqrt{2}}{2}\frac{1}{\modulus{Q_nQ_{n-1}}}\)---the geometric
  mean of radii of two other cycles. This again repeats the relation
  between cycles' radii in the first case.
\end{enumerate}
Three configurations have fixed ratio \(\sqrt{2}\) between respective
horocycles' radii. Thus, they are equally suitable for the proofs based on
the size of horocycles, e.g.~\cite{BeardonShort14a}*{Thm.~4.1}. 

On the other hand, there is a tiny computational advantage in the case
of orthogonal horocycles. Let we have the sequence \(p_j\) of partial
fractions \(p_j=\frac{P_j}{Q_j}\) and want to rebuild the
corresponding chain of horocycles. A horocycle with the point of
contact \(p_j\) has components \((1,p_j, n_j, p_j^2)\), thus only the
value of \(n_j\) need to be calculated at every step. If we use the
condition ``to be tangent to the previous horocycle'', then the quadratic
relation~\eqref{eq:tangent-cycles} shall be solved. Meanwhile, the
orthogonality relation~\eqref{eq:first-ortho-gen} is linear in \(n_j\).

\section{Multi-dimensional M\"obius maps and cycles}
\label{sec:multi-dimens-vari}

It is natural to look for multidimensional generalisations of
continued fractions.  A geometric approach based on M\"obius
transformation and Clifford algebras was proposed
in~\cite{Beardon03a}. 
The Clifford algebra \(\Cliff{n}\) is the associative unital algebra over
\(\Space{R}{}\) generated by the elements \(e_1\),\ldots,\(e_n\)
satisfying the following relation:
\begin{displaymath}
  e_i e_j + e_je_i=-2\delta_{ij},
\end{displaymath}
where \(\delta_{ij}\) is the Kronecker delta. An element of
\(\Cliff{n}\) having the form \(x=x_1e_1+\ldots+x_ne_n\) can be
associated with vectors \((x_1,\ldots,x_n)\in\Space{R}{n}\).  The
\emph{reversion} \(a\mapsto a^*\) in
\(\Cliff{n}\)~\cite{Cnops02a}*{(1.19(ii))} is defined on vectors by
\(x^*=x\) and extended to other elements by the relation
\((ab)^*=b^*a^*\). Similarly the \emph{conjugation} is defined on
vectors by \(\bar{x}=-x\) and the relation
\(\overline{ab}=\bar{b}\bar{a}\). We also use the notation
\(\modulus{a}^2=a\bar{a}\geq 0\) for any product \(a\) of vectors.
An important observation is that any non-zero vectors \(x\) has a
multiplicative inverse: \(x^{-1}=\frac{\bar{x}}{\modulus{x}^2}\).

By Ahlfors~\cite{Ahlfors86} (see
also~\citelist{\cite{Beardon03a}*{\S~5} \cite{Cnops02a}*{Thm.~4.10}})
a matrix \(M=
\begin{pmatrix}
  a&b\\c&d
\end{pmatrix}\) with Clifford entries defines a linear-fractional transformation of
\(\Space{R}{n}\) if the
following conditions are satisfied:
\begin{enumerate}
\item \(a\), \(b\), \(c\) and \(d\) are products of vectors in
  \(\Space{R}{n}\);
\item\label{it:ab-cd-ca-db-vectors} \(ab^*\), \(cd^*\), \(c^*a\) and \(d^*b\) are vectors in
  \(\Space{R}{n}\);
\item the pseudodeterminant \(\delta:=ad^*-bc^*\) is a non-zero real number.
\end{enumerate}
Clearly we can scale the matrix to have the pseudodeterminant
\(\delta=\pm 1\) without an effect on the related linear-fractional
transformation. Define, cf.~\cite{Cnops02a}*{(4.7)}
\begin{equation}
  \label{eq:matrix-bar-star}
  \bar{M}=
\begin{pmatrix}
  d^*&-b^*\\-c^*&a^*
\end{pmatrix}\qquad \text{ and } \qquad
M^*=\begin{pmatrix}
  \bar{d} &\bar{b}\\\bar{c}&\bar{a}
\end{pmatrix}.
\end{equation}
Then \(M\bar{M}=\delta I\) and \(\bar{M}=\kappa M^*\), where
\(\kappa=1\) or \(-1\) depending either \(d\) is a product of even or
odd number of vectors.

To adopt the discussion from Section~\ref{sec:mobi-transf-cycl} to
several dimensions we use vector rather than paravector
formalism, see~\cite{Cnops02a}*{(1.42)} for a discussion. 
Namely, we consider vectors \(x\in\Space{R}{n+1}\) as elements
\(x=x_1e_1+\ldots+x_ne_n+x_{n+1} e_{n+1}\) in \(\Cliff{n+1}\). 
Therefore we can extend the M\"obius
transformation defined by \(M=
\begin{pmatrix}
  a&b\\c&d
\end{pmatrix}\) with \(a,b,c,d\in\Cliff{n}\) to act on
\(\Space{R}{n+1}\). Again, such transformations commute with the
reflection \(R\) in the hyperplane \(x_{n+1}=0\):
\begin{displaymath}
  R:\quad x_1e_1+\ldots+x_ne_n+x_{n+1} e_{n+1}\quad
  \mapsto \quad x_1e_1+\ldots+x_ne_n-x_{n+1} e_{n+1}.
\end{displaymath}
Thus we can consider the M\"obius maps acting on the equivalence
classes \(x\sim R(x)\).

Spheres and hyperplanes in \(\Space{R}{n+1}\)---which we continue to
call cycles---can be associated to
\(2\times 2\) matrices~\citelist{\cite{FillmoreSpringer90a} \cite{Cnops02a}*{(4.12)}}:
\begin{equation}
  \label{eq:spheres-Rn}
  k\bar{x}x-l\bar{x}-x\bar{l}+m=0 \quad \leftrightarrow \quad 
  \cycle{}{}=
  \begin{pmatrix}
    l & m\\
    k & \bar{l}
  \end{pmatrix},
\end{equation}
where \(k, m\in\Space{R}{}\) and \(l\in\Space{R}{n+1}\).  For brevity
we also encode a cycle by its coefficients \((k,l,m)\).  A
justification of~\eqref{eq:spheres-Rn} is provided by the identity:
\begin{displaymath}
  \begin{pmatrix}
    1&\bar{x}
  \end{pmatrix}
  \begin{pmatrix}
    l & m\\
    k & \bar{l}
  \end{pmatrix}
  \begin{pmatrix}
    {x}\\1
  \end{pmatrix}=
  kx\bar{x}-l\bar{x}-x\bar{l}+m,\quad \text{ since } \bar{x}=-x \text{
    for } x\in\Space{R}{n}. 
\end{displaymath}
The identification is also M\"obius-covariant in the sense that the
transformation associated with the Ahlfors matrix \(M\) send a cycle
\(\cycle{}{}\) to the cycle \(M\cycle{}{}M^{*}\)~\cite{Cnops02a}*{(4.16)}.
The equivalence  \(x\sim R(x)\) is extended to spheres:
\begin{displaymath}
  \begin{pmatrix}
    l & m\\
    k & \bar{l}
  \end{pmatrix}\quad\sim   \quad
  \begin{pmatrix}
    R(l) & m\\
    k & R(\bar{l})
  \end{pmatrix}
\end{displaymath}
since it is preserved by the M\"obius transformations with
coefficients from \(\Cliff{n}\). 

Similarly to~\eqref{eq:cycle-iner-pr} we define the M\"obius-invariant
inner product of cycles by the identity
\(\scalar{\cycle{}{}}{\cycle[\tilde]{}{}}=\Re
\tr(\cycle{}{}\cycle[\tilde]{}{})\), where \(\Re\) denotes the scalar
part of a Clifford number.  The orthogonality condition
\(\scalar{\cycle{}{}}{\cycle[\tilde]{}{}}=0\) means that the
respective cycle are geometrically orthogonal in
\(\Space{R}{n+1}\). 

\section{Continued fractions from Clifford algebras and horocycles}
\label{sec:mult-cont-fract}

There is an association between the triangular matrices and the elementary
M\"obius maps of \(\Space{R}{n}\), cf.~\eqref{eq:cf-moebius-maps}:
\begin{equation}
  \label{eq:clifford-inversion}
  \begin{pmatrix}
    0&1\\1&b
  \end{pmatrix}: \
  x\ \mapsto\ (x+b)^{-1}\,,\ 
  \text{ where }
  x=x_1e_1+\ldots+x_ne_n, b=b_1e_1+\ldots b_ne_n,
\end{equation}
Similar to the real line case in
Section~\ref{sec:continued-fractions}, Beardon
proposed~\cite{Beardon03a} to consider the composition of a series of
such transformations as multidimensional continued fraction,
cf.~\eqref{eq:cf-moebius-maps}. It can be again represented as the the
product~\eqref{eq:cf-part-frac-matrix} of the respective \(2\times 2\)
matrices. Another construction of multidimensional continued fractions
based on horocycles was hinted in~\cite{BeardonShort14a}. We wish to
clarify the connection between them.  The bridge is provided by the
respective modifications of Lem.~\ref{le:first-col}--\ref{le:blue}.
\begin{lem}
  \label{le:first-col-mul}
  The cycles \((0,e_{n+1},m)\) (the ``horizontal'' hyperplane
  \(x_{n+1}=m\)) are the only cycles, such that their images under the
  M\"obius transformation \(\begin{pmatrix} a&b\\c&d
  \end{pmatrix}\) are independent from the column \(
  \begin{pmatrix}
    b\\d
  \end{pmatrix}\). The image 
  associated to the column \(\begin{pmatrix} a\\c
  \end{pmatrix}\) is the horocycle \((-m\modulus{c}^2, -ma\bar{c}+\delta  e_{n+1},m\modulus{a}^2)\), which
  touches the hyperplane \(x_{n+1}=0\) at \(\frac{a\bar{c}}{\modulus{c}^2}\) and has the radius
  \(\frac{1 }{m\modulus{c}^2}\).
\end{lem}
\begin{lem}
  \label{le:second-col-mul}
  The cycles \((k,e_{n+1},0)\) (with the equation \(k(u^2+v^2)-2v=0\)) are
  the only cycles, such that their images under the M\"obius
  transformation \(\begin{pmatrix} a&b\\c&d
  \end{pmatrix}\) are independent from the column \(
  \begin{pmatrix}
    a\\c
  \end{pmatrix}\). The image 
  associated to the column  \(\begin{pmatrix} b\\d
  \end{pmatrix}\) is the horocycle
  \((k\modulus{d}^2,kb\bar{d}+\delta e_{n+1},-kb \bar{b})\), which touches the
  hyperplane \(x_{n+1}=0\) at \(\frac{b\bar{d}}{\modulus{d}^2}\) and has
  the radius \(\frac{1 }{k\modulus{d}^2}\).
\end{lem}
The proof of the above lemmas are reduced to multiplications of
respective matrices with Clifford entries. 
\begin{lem}
\label{le:blue-multi}
A cycle \(\cycle{}{}=(0,l,0)\), where \(l=x+re_{n+1}\) and \(0\neq
x\in\Space{R}{n}\), \(r\in\Space{R}{}\), that is any non-horizontal
hyperplane passing the origin, is transformed into
\(M\cycle{}{}M^*=(c x\bar{d}+d\bar{x}\bar{c}, a
x\bar{d}+b\bar{x}\bar{c}+\delta r e_{n+1}, a x\bar{b}+
b\bar{x}\bar{a})\).  This cycle passes points
\(\frac{a\bar{c}}{\modulus{c}^2}\) and \(\frac{b\bar{d}}{\modulus{d}^2}\).  

If \(x= \bar{c}d\), then the centre of
\begin{displaymath}
  M\cycle{}{}M^*=(2\modulus{c}^2\modulus{d}^2,
a\bar{c}\modulus{d}^2+b\bar{d}\modulus{c}^2, (a\bar{c})(d\bar{b})+(b\bar{d})(c\bar{a}))
\end{displaymath}
is
\(\frac{1}{2}(\frac{a\bar{c}}{\modulus{c}^2}+\frac{b\bar{d}}{\modulus{d}^2})+
\frac{\delta r}{2\modulus{c}^2\modulus{d}^2}e_{n+1}\), that is, the centre
belongs to the two-dimensional plane passing the points
\(\frac{a\bar{c}}{\modulus{c}^2}\) and \(\frac{b\bar{d}}{\modulus{d}^2}\)
and orthogonal to the hyperplane \(x_{n+1}=0\).
\end{lem}
\begin{proof}
  We note that \(e_{n+1} x=-x e_{n+1}\) for all
  \(x\in\Space{R}{n}\). Thus, for a product of  vectors \(d\in\Cliff {n}\)
  we have \(e_{n+1} \bar{d}=d^* e_{n+1}\). Then  
  \begin{displaymath}
    c e_{n+1}\bar{d}+d\bar{e}_{n+1}\bar{c}=(cd^*-dc^*)e_{n+1}=(cd^*-(cd^*)^*)e_{n+1}=0,
  \end{displaymath}
  due to the Ahlfors
  condition~\ref{it:ab-cd-ca-db-vectors}. Similarly, \(a
  e_{n+1}\bar{b}+ b\bar{e}_{n+1}\bar{a}=0\) and \(a
  e_{n+1}\bar{d}+b\bar{e}_{n+1}\bar{c}=(ad^*-bc^*) e_{n+1}=\delta e_{n+1}\). 

  The image \(M\cycle{}{}M^*\) of the cycle \(\cycle{}{}=(0,l,0)\) is
  \((c l\bar{d}+d\bar{l}\bar{c}, a l\bar{d}+b\bar{l}\bar{c}, a
  l\bar{b}+ b\bar{l}{a})\).  From the above calculations for
  \(l=x+re_{n+1}\) it becomes \((c x\bar{d}+d\bar{x}\bar{c}, a
  x\bar{d}+b\bar{x}\bar{c}+\delta r e_{n+1}, a x\bar{b}+
  b\bar{x}\bar{a})\). The rest of statement is verified by the substitution.
\end{proof}

Thus, we have exactly the same freedom to choose representing
horocycles as in Section~\ref{sec:cont-fract-cycl}: make two
consecutive horocycles either tangent or orthogonal.  To visualise
this, we may use the two-dimensional plane \(V\) passing the points of
contacts of two consecutive horocycles and orthogonal to
\(x_{n+1}=0\). It is natural to choose the connecting cycle (drawn in
blue on Fig.~\ref{fig:vari-arrang-three}) with the centre belonging to
\(V\), this eliminates excessive degrees of freedom. The corresponding
parameters are described in the second part of
Lem.~\ref{le:blue-multi}.
Then, the intersection of horocycles with \(V\)
are the same as on Fig.~\ref{fig:vari-arrang-three}. 

Thus, the continued fraction with the partial quotients
\(\frac{P_n\bar{Q}_n}{\modulus{Q_n}^2}\in\Space{R}{n}\) can be
represented by the chain of tangent/orthogonal horocycles. The
observation made at the end of Section~\ref{sec:cont-fract-cycl} on
computational advantage of orthogonal horocycles remains valid in
multidimensional situation as well.

As a further alternative we may shift the focus from horocycles to the
connecting cycle (drawn in blue on
Fig.~\ref{fig:vari-arrang-three}). The part of the space
\(\Space{R}{n}\) encloses inside the connecting cycle is the image
under the corresponding M\"obius transformation of the half-space of
\(\Space{R}{n}\) cut by the hyperplane \((0,l,0)\) from
Lem.~\ref{le:blue-multi}. Assume a sequence of connecting cycles
\(\cycle{}{j}\) satisfies the following two conditions, e.g. in
Seidel--Stern-type theorem~\cite{BeardonShort14a}*{Thm~4.1}:
\begin{enumerate}
\item for any \(j\), the cycle \(\cycle{}{j}\) is enclosed within the cycle
  \(\cycle{}{j-1}\);
\item the sequence of radii of \(\cycle{}{j}\) tends to zero.
\end{enumerate}
Under the above assumption the sequence of partial fractions
converges. Furthermore, if we use the connecting cycles in the third
arrangement, that is generated by the cycle  \((0,x+e_{n+1},0)\),
where \(\norm{x}=1\), \(x\in\Space{R}{n}\), then the above second condition
can be replaced by following 
\begin{enumerate}
\item[(\(2'\))] the sequence of \(x_{n+1}^{(j)}\) of \((n+1)\)-th
  coordinates of the centres of the connecting cycles \(\cycle{}{j}\)
  tends to zero.
\end{enumerate}
Thus, the sequence of connecting cycles is a useful tool to describe
a continued fraction even without a relation to horocycles.

Summing up, we started from multidimensional continued fractions
defined through the composition of M\"obius transformations in
Clifford algebras and associated to it the respective chain of
horocycles. This establishes the equivalence of two approaches proposed
in~\cites{Beardon03a} and \cite{BeardonShort14a} respectively.



{\small
\providecommand{\noopsort}[1]{} \providecommand{\printfirst}[2]{#1}
  \providecommand{\singleletter}[1]{#1} \providecommand{\switchargs}[2]{#2#1}
  \providecommand{\irm}{\textup{I}} \providecommand{\iirm}{\textup{II}}
  \providecommand{\vrm}{\textup{V}} \providecommand{\cprime}{'}
  \providecommand{\eprint}[2]{\texttt{#2}}
  \providecommand{\myeprint}[2]{\texttt{#2}}
  \providecommand{\arXiv}[1]{\myeprint{http://arXiv.org/abs/#1}{arXiv:#1}}
  \providecommand{\doi}[1]{\href{http://dx.doi.org/#1}{doi:
  #1}}\providecommand{\CPP}{\texttt{C++}}
  \providecommand{\NoWEB}{\texttt{noweb}}
  \providecommand{\MetaPost}{\texttt{Meta}\-\texttt{Post}}
  \providecommand{\GiNaC}{\textsf{GiNaC}}
  \providecommand{\pyGiNaC}{\textsf{pyGiNaC}}
  \providecommand{\Asymptote}{\texttt{Asymptote}}

}
\end{document}